\documentclass[a4paper,12pt]{article}
\usepackage{amsmath,amsthm,geometry,mathrsfs,hyperref,multicol,graphtex3.41}
\newtheorem{thm}{Theorem}[section]
\theoremstyle{definition}
\newtheorem{defn}{Definition}[section]
\newtheorem{exmp}{Example}[section]

\newcommand\cl{\mbox{cl}}
\newbox\bdots
\setbox\bdots=\hbox{\pic \SetUnits[pt] (3.4,3.4,3.4) \PlotSize5 \Plot(0,0) (-1,1) (1,-1) \cip}
\newbox\fdots
\setbox\fdots=\hbox{\pic \SetUnits[pt] (3.4,3.4,3.4) \PlotSize 5 \Plot(0,0) (1,1) (-1,-1) \cip}
\title{A Fundamental Theorem on Graph Operators}
\author{Severino V.~Gervacio \\
Department of Mathematics and Statistics\\
De La Salle University\\
0922 Manila, Philippines}
\date{}
\begin{document}
\maketitle

\begin{abstract}
A graph operator is a function $\Gamma$ defined on some set of graphs such that whenever two graphs $G$ and $H$ are isomorphic, written $G\simeq H$, then $\Gamma(G)\simeq \Gamma(H)$. For a graph $G$ not in the domain of $\Gamma$, we put $\Gamma(G)=\emptyset$.  Also, let us define $\Gamma^0(G)=G$, and for any integr $k\ge1$, $\Gamma^k(G)=\Gamma(\Gamma^{k-1}(G))$

We prove that if $\Gamma$ is a graph operator, then the sequence $\langle \Gamma^k(G)\rangle_{k=0}^\infty$ has only three possible types of behaviour.  Either $\Gamma^k(G)=\emptyset$ for some integer $k>0$, or $\displaystyle\lim_{k\to\infty}|V(\Gamma^k(G))|=\infty$, or there exist integers $m\ge0$, $p>0$ such that the graphs $\Gamma^j(G)$ are non-isomorphic ($0\le j\le m)$, and $\Gamma^{n+p}\simeq \Gamma ^n(G)$ for all integers $n\ge m$. We illustrate this using two new graph operators, namely, the path graph operator and the claw graph operator.\\

		
\end{abstract}

\section{Introduction}
A (simple) \emph{graph} $G$ is an ordered pair $G=\langle V(G), E(G)\rangle$ where $V(G)$ is a  non-empty finite set of elements called \emph{vertices} and $E(G)$ is a set (possibly empty) of 2-subsets of $V(G)$, called \emph{edges}.  Whenever convenient, we shall denote an edge $\{a,b\}$ in $E(G)$ by $ab$ or $ba$.

The notations $P_n, C_n, K_n$ are used to denote the path of order $n$, the cycle of length $n$, and the complete graph of order $n$, respectively.  

The symbol $\overline G$ denotes the complement of the graph $G$.  .

If two graphs $G$ and $H$ are isomorphic, we shall write $G\simeq H$.

Readers may refer to the book by Harary \cite{Harary} for the definitions of concepts that are not given in this paper.

\section{Graph Operators}
Some graph operators will be mentioned in this paper and two new ones will be defined and discussed to illustrate a fundamental theorem on graph  operators. Let us define formally a graph operator.

\begin{defn}
A \emph{graph operator} is a function $\Gamma$ defined on a set of graphs such that $\Gamma(G)\simeq \Gamma(H)$ whenever $G\simeq H$.
\end{defn}
 
If $\Gamma$ is a graph operator, we put $\Gamma^0(G)=G$ and for each integer $k\ge1$, let $\Gamma^k(G)=\Gamma(\Gamma^{k-1}(G)$.  If $G$ is not in the domain of $\Gamma$, let us put $\Gamma(G)=\emptyset$.  We also define $\Gamma(\emptyset)=\emptyset$.  Thus, the sequence $\langle \Gamma^k(G)\rangle_{k=0}^\infty$ is well-defined.  This paper studies the behaviour of this sequence.

A classic graph operator is the \emph{line graph operator} defined for all graphs with at least one edge.  If $G$ is a graph with at least one edge, the line graph of $G$, denoted by $\mathcal L(G)$, is the graph whose vertices are the edges of $G$ and where two vertices $e_1$ and $e_2$ in $V(\mathcal L(G))$ are adjacent if and only if $e_1\ne e_2$ and $e_1\cap e_2\ne\emptyset$.  Please refer to Figure \ref{fig:linegraph} for an illustration of line graph.

\begin{figure}[h]
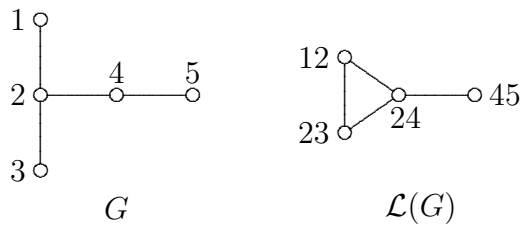

$$\pic
\SetUnits[cm] (1,1,1)
\Path (0,1) (0,0) (1,0) (2,0)
\Vertex(0,-1)
\Edge(0,-1) (0,0)
\Align[c] (1) (-0.3,1)
\Align[c] (2) (-0.3,0)
\Align[c] (3) (-0.3,-1)
\Align[c]  (4) (1,0.3)
\Align[c] (5) (2,0.3)
\Align[c] ($G$) (1,-1.5)
 
\Path(4,0.5) (4,-0.5) (4.71,0) (5.71,0)
\Edge(4,0.5) (4.71,0)
\Align[r] (12) (3.8,0.5)
\Align[r] (23) (3.8,-0.5)
\Align[c] (24) (4.8,-0.3)
\Align[l] (45) (5.9,0)
\Align[c] ($\mathcal L(G)$) (5,-1.5)  
\cip$$
\caption{A graph $G$ and its line graph $\mathcal L(G)$}
\label{fig:linegraph}
\end{figure}

Krausz \cite{krausz} showed that a graph $G$ is a line graph, \emph{i.e.}, $\mathcal L(H)\simeq G$ for some graph $H$, if and only if the edges of $G$ can be partitioned into cliques (complete subgraphs) such that each vertex of $G$ belongs to only one or two cliques.

Let us set $\mathscr L^0(G)=G$.  If $G$ has no edges, let us write $\mathscr L(G)=\emptyset$.  Put $\mathscr L(\emptyset)=\emptyset$. For any integer $k\ge1$, define $\mathcal L^k(G)=\mathcal L(\mathcal L^{k-1}(G)$) and consider the sequence of graphs $\langle G, \mathcal L(G), \mathcal L^2(G), \mathcal L^3(G), \cdots\rangle$. Van Rooij and Wilf \cite {vanRooij} showed that when $G$ is any connected graph, only four behaviours are possible for this sequence:   
	\begin{enumerate}
		\item If $G\simeq C_n$ is a cycle, then every element of the sequence is isomorphic to $C_n$.  These are the only connected graphs for which $\mathcal L(G)\simeq G$.  
		\item If $G\simeq K_{1,3}$, then $\mathcal L^k(G)\simeq C_3$ for all $k\ge1$.  
		\item If $G\simeq P_n$ is a path, then each subsequent graph in the sequence is a path one unit shorter than the previous graph and the sequence terminates with a null graph.  
		\item In all remaining cases, the sizes of the graphs in the sequence eventually increases without bound.
	\end{enumerate}

Let us define formally more concept related to graph operators. 

\begin{defn}
Let $\Gamma$ be a graph operator and $G$ a graph in the domain of $\Gamma$. The graph $G$ is
\begin{enumerate}
\item \emph{$\Gamma$-vanishing} if $\Gamma^k(G)=\emptyset$ for some integer $k>0$;  
\item \emph{$\Gamma$-expanding} if $\displaystyle\lim_{k\to\infty}|V(\Gamma^k(G))|=\infty$;  
\item \emph{$\Gamma$-periodic }if $\Gamma^k(G)\ne\emptyset$ for all integers $k\ge0$, and there exist integers $m\ge0$, $p>0$ such that  $\Gamma^{n+p}(G)\simeq \Gamma^n(G)$ for all integers $n\ge m$.  The smallest $m$ and $p$ satisfying these conditions are called the \emph{tail length} and \emph{period} of $G$, respectively.
\end{enumerate}
\end{defn}
 
We see then that cycles are $\mathcal L$-periodic.  Here, the tail length is $m=0$ and the period is $p=1$.  The claw $K_{1,3}$ is $\mathcal L$-periodic, with tail length $m=1$, and period $p=1$.

The paths $P_n$ are $\mathcal L$-vanishing, and all other connected graphs different from $C_n$, $K_{1,3}$, $P_n$ are $\mathcal L$-expanding.

If a graph $G$ has more than one component, it is still true that $G$ is either $\mathcal L$-vanishing, or $\mathcal L$-expanding, or $\mathcal L$-periodic. 

Another example of a graph operator is the \emph{triangle graph} of a graph apparently introduced independently by Pullman \cite{Pullman}, Egawa and Ramos \cite{EgawaRamos}, and Balakrishnan \cite{Balakrishnan}. The triangle graph of a graph, denoted by $\mathscr T(G)$ is the graph whose vertices are the triangles (cycle of length 3) in $G$ and two vertices in $\mathscr T(G)$ are adjacent if and only if they are distinct triangles sharing a common edge.
	
	The triangle graph of a graph $G$ is an induced subgraph of the \emph{cycle graph} $\mathscr C(G)$ introduced in 1983 \cite{gervacio, Egawa, Seema}.  The vertices of $\mathscr C(G)$ are all the induced cycles in $G$ and two vertices in $\mathscr C(G)$ are adjacent if and only if they are distinct induced cycles sharing at least one common edge.
	
	In both the triangle graph operator and the cycle graph operator, graphs that are  vanishing, periodic, and expanding  are known to exist.
	
\newpage

\section{Main Result}

We now state the main result of this paper.  
\begin{thm}[Fundamental Theorem on Graph Operators]
Let $\Gamma$ be a graph operator, and $G$ a graph in the domain of $\Gamma$.  Then $G$ is either $\Gamma$-vanishing, or $\Gamma$-expanding, or $\Gamma$-periodic.
\end{thm}

In the next two sections, we shall introduce two new graph operators and use them to illustrate the fundamental theorem.  The proof of the fundamental theorem will be given in a later section

\section{Path Graph Operator}
We have illustrated the fundamental theorem on graph operators using the line graph operator.  

We shall give another illustration using the path graph operator.
\begin{defn}
Let $G$ be a graph with at least one edge.  The \emph{path graph of $G$}, denoted by $\mathscr P(G)$ is the graph whose vertices are the induced paths in $G$ of length at least 1.  Two vertices in $\mathscr P(G)$ are adjacent if they are distinct induced paths sharing at least one edge.
\end{defn}
 
Figure \ref{fig:pathgraph} shows an example of a graph $G$ and its path graph $\mathscr P(G)$.

\begin{figure}[h]
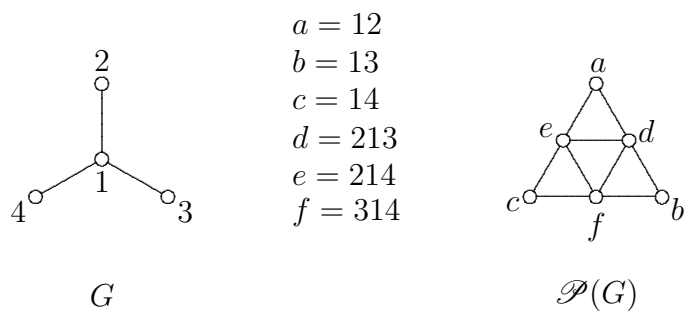

$$\pic
\SetUnits[cm] (1,1,1)
\Align[c] ($\mathscr P(G)$) (0,-1.8)
\Path(-0.87,-0.5) (0,-0.5)
\Edge(0,-0.5) (0.87,-0.5) (0,-0.5) (0.43,0.25)
\Align[c] ($c$) (-1.1,-0.6)
\Align[c] ($f$) (0,-0.9)
\Rotate[120] (0,0)
\Path(-0.87,-0.5) (0,-0.5)
\Edge(0,-0.5) (0.87,-0.5) (0,-0.5) (0.43,0.25)
\Align[c] ($b$) (-1.1,-0.6)
\Align[c] ($d$) (0,-0.76)
\Rotate[-120] (0,0)
\Path(-0.87,-0.5) (0,-0.5)
\Edge(0,-0.5) (0.87,-0.5) (0,-0.5) (0.43,0.25)
\Align[c] ($a$) (-1.1,-0.6)
\Align[c] ($e$) (0,-0.76)
\offrotate
\Align[l] ($a=12$) (-4,1.8) 
\Align[l] ($b=13$) (-4,1.3)  
\Align[l] ($c=14$) (-4,0.8)   
\Align[l] ($d=213$) (-4,0.3) 
\Align[l] ($e=214$) (-4,-0.2) 
\Align[l] ($f=314$) (-4,-0.7) 
\Translate(0,0) (-6.5,0)
\Align[c] ($G$) (0,-1.8)
\Path(0,1) (0,0) (0.87,-0.5)
\Vertex(-0.87,-0.5)
\Edge(0,0) (-0.87,-0.5)
\Align[c] (1) (0,-0.3)
\Align[c] (2) (0,1.3)
\Align[c] (3) (1.1,-0.7)
\Align[c] (4) (-1.1,-0.7)
\cip$$
\caption{A graph $G$ and its path graph $\mathscr P(G)$}
\label{fig:pathgraph}
\end{figure}

In \cite{jcdcggg} it was shown that a graph $G$ in the domain of $\mathscr P$ can only be either $\mathscr P$-vanishing, or $\mathscr P$-expanding, or $\mathscr P$-periodic.  More precisely, the following result was established in that paper.  

\newpage
\begin{thm}
A graph $G$ with at least one edge is
\begin{enumerate}
\item $\mathscr P$-vanishing if and only if every component of $G$ is a complete graph;
\item $\mathscr P$-expanding if and only if $G$ contains two induced paths isomorphic to $P_3$ that share a common edge;
\item $\mathscr P$-periodic if and only if exactly one component of $G$ is a graph isomorphic to $K_q-e$, $q\ge3$, and all other components are complete graphs.
\end{enumerate}
\end{thm} 
 
The graph $G$ in Figure \ref{fig:pathgraph} has two induced paths $213$ and $314$, both isomorphic to $P_3$, and they share the edge $13$.  Therefore, according to the above theorem, $G$ is $\mathscr P$-expanding.  This means that if we apply the operator $\mathscr P$ iteratively the order of the graph grows larger and larger without bound.

\section{Claw Graph Operator}
An induced subgraph in a graph that is isomorphic to $K_{1,3}$ is called a claw \cite{King}
\begin{defn}
The \emph{claw graph} of a graph $G$, denoted by $\mbox  {cl}\,(G)$, is the graph whose vertices are the claws in $G$, where two vertices are adjacent if and only if they are distinct claws sharing a common edge.
\end{defn}

\begin{exmp}{Planar grids $P_n\times P_2$ are cl-vanishing graphs}
\begin{figure}[h]
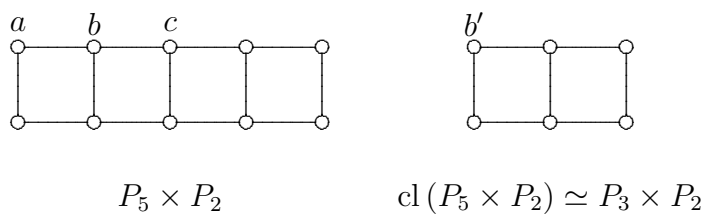

$$\pic
\SetUnits[cm] (1,1,1)
\SimpleVertex{black}
\PlanarGrid(5,2) [1cm] (0,0)
\Align[c] ($a$) (-2,0.8)
\Align[c] ($b$) (-1,0.8)
\Align[c] ($c$) (0,0.8)
\Align[c] ($P_5\times P_2$) (0,-1.5)
\Align[c] ($b'$) (4,0.8)
\PlanarGrid(3,2) [1cm] (5,0)
\Align[c] (cl$\,(P_5\times P_2)\simeq P_3\times P_2$) (5,-1.5)
\cip$$
\caption{Claw-vanishing planar grid}
\label{fig:ladder}
\end{figure}
\end{exmp}

\begin{exmp}{The Petersen graph is cl-periodic, with period 1}
\begin{figure}[h]
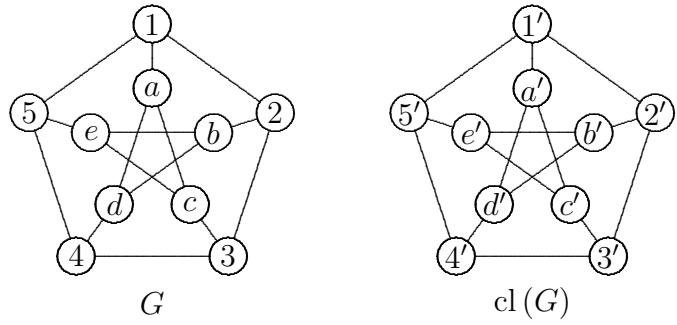

$$\pic
\SetUnits[cm] (1,1,1)
\VertexRadius7pt
\SimpleVertex{black}
\Petersen(5,2) [2cm] (0,0)
\Align[c] ($G$) (0,-2)
\Align[c] (1) (0,1.7)
\Align[c] ($a$) (0,0.84)
\Rotate[-72] (0,0)
\Align[c] (2) (0,1.7)
\Align[c] ($b$) (0,0.84)
\Rotate[-144] (0,0)
\Align[c] (3) (0,1.7)
\Align[c] ($c$) (0,0.84)
\Rotate[144] (0,0)
\Align[c] (4) (0,1.7)
\Align[c] ($d$) (0,0.84)
\Rotate[72] (0,0)
\Align[c] (5) (0,1.7)
\Align[c] ($e$) (0,0.84)
\offrotate
\Petersen(5,2) [2cm] (5,0)
\Align[c] (cl$\,(G)$) (5,-2)
\Align[c] ($1'$) (5,1.7)
\Align[c] ($a'$) (5,0.84)
\Rotate[-72] (5,0)
\Align[c] ($2'$) (5,1.7)
\Align[c] ($b'$) (5,0.84)
\Rotate[-144] (5,0)
\Align[c] ($3'$) (5,1.7)
\Align[c] ($c'$) (5,0.84)
\Rotate[144] (5,0)
\Align[c] ($4'$) (5,1.7)
\Align[c] ($d'$) (5,0.84)
\Rotate[72] (5,0)
\Align[c] ($5'$) (5,1.7)
\Align[c] ($e'$) (5,0.84)
\cip$$
\caption{The Petersen graph}
\label{fig:Petersen}
\end{figure}
\end{exmp}
 
In Figure \ref{fig:Petersen}, $1'$ denotes the unique claw having 1 as the central vertex. generally, for a vertex $x$ in the graph $G$, we denote by $x'$ the unique claw having $x$ as the central vertex.
\newpage
\begin{thm}
If $G$ is a triangle-free cubic graph, then cl$\,(G)\simeq G$. \label{thm:periodic}
\end{thm}

\begin{proof}
Let $G$ be a triangle-free cubic graph.  If $x$ is any vertex of $G$, then $x$ together with its three neighbors induce a claw with central vertex at $x$ since $G$ has no triangles (subgraph isomorphic to the cycle $C_3$). Let us denote this claw by $x'$.  Note that $V(\cl(G))=\{x'|x\in V(G)\}$.  This follows from the fact that $G$ is 3-regular.  Now, if $x$ and $y$ are adjacent vertices in $G$, then $x,$ and $y'$ ar e claws in $G$ sharing the edge $xy$.  Therefore, $x'y'$ is an edge of $\cl(G)$.  Conversely, if $x'y'$ is an edge in $\cl(G)$, then $x'$ and $y'$ must be claws in $G$ sharing an edge.  Now, $x$ and $y$ are central vertices of claws in $G$.  If they share a common edge, it must be the edge $xy$. So, we see that two vertices $x$ and $y$ are adjacent in $G$ if and only if the claws $x'$ and $y'$ are adjacent in $\cl(G)$.  This means that the mapping $x\mapsto x'$ is an isomorphism of $G$ onto $\cl(G)$.
\end{proof}
 
Thus, we have an infinite class of cl-periodic graphs with period 1.  In particular, all triangle-free generalized Petersen graphs \cite{GeneralizedPetersen} are cl-periodic.\\
 
\begin{exmp}{Some Generalized Petersen Graphs}
$$\pic
\SetUnits[cm] (1,1,1)
\SolidVertex{xanadu}
\Petersen(7,2) [1.3cm] (0,0)
\Align[c] ($PG(7,2)$) (0,-2)
 
\Petersen(9,4) [1cm] (3.95,0)
\Align[c] ($PG(9,4)$) (3.95,-2)
 
\Petersen(11,3) [0.8cm] (8,0)
\Align[c] ($PG(11,3)$) (8,-2)
\cip$$
\end{exmp}

\begin{thm}
Let $G$ be a triangle-free cubic graph. Obtain a graph $H$ from $G$ by adding a new vertex $x$.  Join $x$ by an edge to any vertex $y$ in $G$ and make $x$ adjacent to any two neighbors of $x$ in $G$.  Then $H$  is $\mbox{\rm cl}$-expanding.
\end{thm}

\begin{proof}
 Let $G$ be a triangle-free cubic graph.  Add a new vertex $x$ and join it by an edge to a vertex, say $y$, in $G$.  Make $x$ adjacent to exactly two neighbors of $y$ in $G$.  This is illustrated in Figure \ref{fig:hat}.   We shall call $x$ a \emph{hat} of $G$. 
\begin{figure}[h]
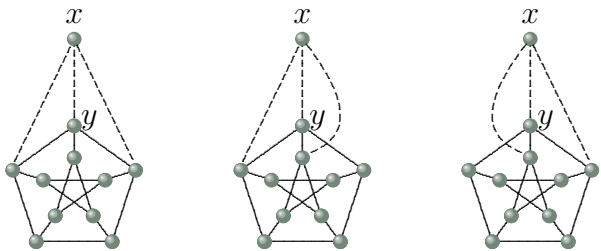

$$\pic
\SetUnits[cm] (1,1,1)
\SolidVertex{xanadu}
\Petersen(5,2) [1cm] (0,0) 
\Vertex(0,2)
\Petersen(5,2) [1cm] (3,0)
\Vertex(3,2)
\Petersen(5,2) [1cm] (6,0)
\Vertex(6,2)
\ondashes
\Edge(0,2) (0,0.88)
\Edge(0,2) (-0.82,0.3) (0,2) (0.82,0.3)
\Edge(3,2) (3,0.88)
\Edge(3,2) (2.18,0.3)
\CEdge(3,2) (3.5,1) (3,0.46)
\Edge(6,2) (6,0.88) (6,2) (6.82,0.3)
\CEdge(6,2) (5.5,1) (6,0.46)
\Align[c] ($x$) (0,2.3) (3,2.3) (6,2.3)
\Align[c] ($y$) (0.2,0.92) (3.2,0.92) (6.2,0.92)
\cip$$
\caption{Adding a hat to a triangle-free cubic graph.}
\label{fig:hat}
\end{figure}

Let $H$ denote the graph resulting from $G$ by adding a hat.  We claim that $H$ is a $\cl$-expanding graph.  Since $G$ is an induced graph of $H$ then $\cl(G)$ is an induced subgraph of $\cl(H)$.  Let us denote by $G'$ the graph $\cl(G)$.  By Theorem \ref{thm:periodic}, $G'\simeq G$.  To form $\cl(H)$, we add the $G'$ the additional clawss arising from the addition of the hat $x$.  Please refer to Figure \ref{fig:addhat}.  The vertices shown in the drawing are not necessarily distinct.  For instance, $a$ can be the same vertex as $y_3$ because no triangle will arise in this case.

\begin{figure}[h]
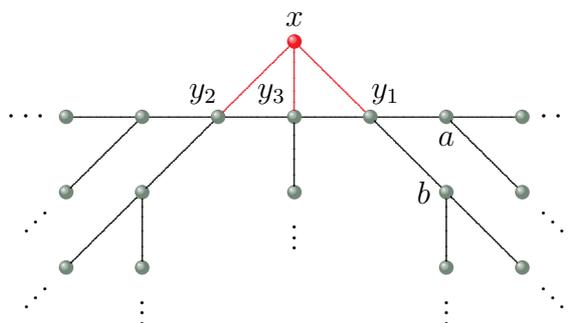

$$\pic
\SetUnits[cm] (1,1,1)
\SolidVertex{xanadu}
\Path(-3,0) (-2,0) (-1,0) (0,0) (1,0) (2,0) (3,0)
\SolidVertex{red}
\PlotColor{red}
\Align[c] ($x$) (0,1.3)
\Align[c] ($y_1$) (1.2,0.3)
\Align[c] ($y_2$) (-1.2,0.3)
\Align[c] ($y_3$) (-0.3,0.3)
\Align[c] ($a$) (2,-.3)
\Align[c] ($b$) (1.7,-1)
\Vertex(0,1) \Edge(0,1) (0,0) (0,1) (1,0) (0,1) (-1,0)
\PlotColor{black}
\SolidVertex{xanadu}
\Vertex(-2,-1) (-3,-2) \Edge(-1,0) (-2,-1) (-2,-1) (-3,-2)
\Vertex(-3,-1) \Edge(-2,0) (-3,-1)
\Vertex(-2,-2) \Edge(-2,-1) (-2,-2)
\Vertex(0,-1) \Edge(0,0) (0,-1)
\Vertex(2,-1) (3,-2) \Edge(1,0) (2,-1) (2,-1) (3,-2)
\Vertex(2,-2) \Edge(2,-1) (2,-2)
\Vertex(3,-1) \Edge(2,0) (3,-1)
\Align[c] ($\cdots$) (3.5,0) (-3.5,0)
\Align[c] ($\vdots$) (0,-1.5) (2,-2.5) (-2,-2.5)
\Align[c] (\copy\bdots) (3.4,-1.4) (3.4,-2.4)
\Align[c] (\copy\fdots) (-3.4,-1.4) (-3.4,-2.4)
\cip$$
\caption{A hat $x$ of a triangle-free cubic graph}
\label{fig:addhat}
\end{figure}

In $G$, the claw $y_1'$ has central vertex $y_1$ and three other vertices $a, b, y_3$. In $H$ there  is another  claw with center at $y_1$ but different from the claw $y_1'$.  This claw has central vertex $y_1$ also and three other vertices $a, b, x$.  Let us denote this by $y_1''$.  The claw $y_1''$ is adjacent to $y_1'$, $a'$, and $b'$ in $\cl(H)$.  Thus $y_1''$ is a hat of $G'$.  Likewise, there is a second  hat $y_2''$ of $G'$.  Hence, $\cl(G)$ is isomorphic to the graph $G$ plus two hats $y_1''$ and $y_2''$.  By iteration, $\cl^2(G)$ is a graph containing a subgraph isomorphic to $G$ plus four hats.  In general, $\cl^k(G)$ contains a subgraph isomorphic to $G$ plus $2^k$ hats.  Hence, $G$ is $\cl$-expanding.
\end{proof}

\section{Proof of Main Result}

Let us now prove the Fundamental Theorem on Graph Operators, Theorem 1.  

 \begin{proof}  Let $\Gamma$ be a graph operator and $G$, a graph in the domain of $\Gamma$.  It can be the case that $\Gamma^k(G)=\emptyset$ just like in the case of the line graph operator.  It can also be the case that $\displaystyle\lim_{k\to\infty}|V(\Gamma^k(G))|=\infty$ just like the line graph operator.  All we need to show is that if $G$ is neither $\Gamma$-vanishing nor $\Gamma$-expanding, then it must be $\Gamma$-periodic.
 
 Consider the sequence $\langle G, \Gamma(G), \Gamma^2(G), \ldots \rangle$.  If $G$ is neither $\Gamma$-vanishing nor $\Gamma$-expanding, then no element of the sequence is $\emptyset$. Furthermore, the order of the graphs in the sequence must be bounded.  Hence, there exists a positive integer $N$ such that $0\le |V(\Gamma^k(G))|\le N$.   

 There are only a finite number of graphs of order not exceeding $N$ up to isomorphism.  Therefore, there are repetitions of elements in the sequence.  This means that some graphs in the sequence are isomorphic.  

Let $\Gamma^m(G)$ be the first graph in the sequence that is repeated.  Then $m\ge0$ and the graphs $\Gamma^k(G)$, $0\le k\le m$ are non-isomorphic.  Let $\Gamma^{m+p}(G)$ be the first repetition of $\Gamma^m(G)$ and let $n\ge m$, say $n=m+t$.  
\begin{align*}
\Gamma^{n+p}(G)&\simeq \Gamma^{m+t+p}(G)\\
&\simeq \Gamma^t(\Gamma^{m+p}(G))\\
&\simeq \Gamma^t(\Gamma^m(G))\\
&\simeq \Gamma^{m+t}(G)\\
&\simeq \Gamma^n(G)
\end{align*}  
Therefore, $G$ is $\Gamma$-periodic. 
\end{proof}

\section{Concluding Remarks}
We know that a graph in the domain of any given graph operator can only be vanishing, expanding, or periodic.  Note however, that there may not exist graphs in some of the three categories. 

Consider the graph operator $\mathcal C(G)=\overline G$, the complement of $G$.  If $G$ is self-complementary, then $G$ is periodic with period 1.  If $G$ is not self-complementary, then $G$ is periodic with period 2 since $\overline{\overline{G}}=G$.  So all graphs are periodic with respect to the operator $\mathcal C$.  There are no vanishing graphs and no expanding graphs with respect to this graph operator.

Let the operator $\mathcal S$ defined as follows:  Given a graph $G$, let $\mathcal S(G)$ be the graph obtained from $G$ by subdividing every edge of $G$.  To subdivide an edge $ab$ means to remove the edge $ab$ from the graph and replace it by two new edges $ax$ and $xb$ where $x$ is a new vertex not in $G$.  For this operator, a graph without any edge is periodic with period 1.  A graph with at least one edge is an expanding graph.  There are no vanishing graphs with respect to this operator.

The shadow graph operator \cite{shadow} assigns to a graph $G$ the graph $D_2(G)$ formed by adding a copy $G'$ of $G$ with vertex-set $V(G')=\{v'|v\in V(G)\}$.  The edges of $D_2(G)$ consists of the edges of $G$ and $G'$ plus all edges of the form $v'x$ where $x$ is a neighbor of $v$, $v\in V(G)$.  Note that the order of $D_2(G)$ is twice the order of $G$.  Thus, every graph is $D_2$-expanding.  There are no vanishing graphs and no periodic graphs with respect to the shadow graph operator.

\end{document}